\documentclass[11pt,draft]{article}
\usepackage{amsmath,amscd}
\usepackage{amssymb,latexsym,amsthm}
\usepackage{color}
\usepackage{amssymb}
\usepackage{color}
\usepackage{amsmath,amsthm,amscd}
\usepackage[latin1]{inputenc}
\usepackage{lscape}
\usepackage{fancyhdr}
\usepackage{amsfonts}
\usepackage{pb-diagram}

\numberwithin{equation}{section}
\newtheorem{theorem}{Theorem}[section]

\newtheorem{proposition}[theorem]{Proposition}
\newtheorem{lemma}[theorem]{Lemma}

\newtheorem{corollary}[theorem]{Corollary}

\theoremstyle{definition}
\newtheorem{example}[theorem]{Example}
\newtheorem{definition}[theorem]{Definition}
\newtheorem{examples}[theorem]{Examples}

\newtheorem{remark}[theorem]{Remark}

\newcommand{\cO}{\mbox{${\cal O}$}}

\title{\textbf{Koszulity for graded skew PBW extensions}}
\author{Héctor Suárez \footnote{Seminario de Álgebra Constructiva - $SAC^2$,
Universidad Nacional de Colombia - sede Bogotá.}  \footnote{Escuela de Matemáticas y Estadística,
Universidad Pedagógica y Tecnológica de Colombia - sede Tunja. }}
\date{}
\begin{document}
\maketitle
\begin{abstract}
\noindent Pre-Koszul and Koszul algebras were defined by Priddy in \cite{Priddy1970}. There exist
some relations between these algebras and the skew PBW  extensions defined in \cite{LezamaGallego}.
In \cite{SuarezReyes2016} we  gave  conditions to guarantee that skew PBW extensions over
fields it turns out homogeneous pre-Koszul or Koszul algebra. In this paper we complement these
results defining graded skew PBW extensions and showing that if $R$ is a finite presented Koszul
$\mathbb{K}$-algebra then every graded skew PBW extension of $R$ is Koszul.

\bigskip


\bigskip

\end{abstract}

\section{Introduction}


Pre-Koszul and Koszul algebras were  introduced by Priddy in \cite{Priddy1970}. Koszul algebras
have important applications in algebraic geometry, Lie theory, quantum groups, algebraic topology
and combinatorics. The  structure and  history of Koszul algebras are  detailed in
\cite{Polishchuk}. There exist numerous equivalent definitions of a Koszul algebra (see for example
\cite{Backelin}). Koszul algebras have been defined in a more general way by some authors (see
for example \cite{BeilinsonGinzburgSoerge1996}, \cite{Berger7}, \cite{Cassidy2008}, \cite{Li2012},
\cite{Woodcock1998}). Other authors have studied some properties of algebras constructed from Koszul algebras (see for example \cite{He5} and \cite{Shepler1}). In this paper we will consider the classical notion of Kozulity introduced by
Priddy.

Skew PBW extensions or $\sigma$-PBW extensions were defined in \cite{LezamaGallego}. Several
properties of these extensions have been recently studied (see for example \cite{Artamonov}, \cite{LezamaReyes},
\cite{LezamaAcostaReyes2015}, \cite{Reyes2013}, \cite{Reyes2014},
\cite{Reyes2014UIS}, \cite{Reyes2015}, \cite{SuarezLezamaReyes2015}, \cite{Venegas2015}). There
exist some relations between Koszul algebras with the skew PBW extensions of fields. In
\cite{SuarezReyes2016} we prove that every semi-commutative skew PBW extension of a field is
Koszul. In the literature there exist examples of  Koszul algebras  which are skew PBW extensions
of a $\mathbb{K}$-algebra $R\neq \mathbb{K}$. For example, the Jordan plane  is an Artin-Schelter
regular algebra of dimension two and therefore it is a Koszul algebra, but the Jordan plane is a
PBW extension of  $\mathbb{K}[x]$. Therefore, the results given in \cite{SuarezReyes2016}
does not apply in this case. We define graded skew PBW extensions and showed  that every graded skew PBW extension of  a finitely presented Koszul algebra  is Koszul. Thus, our interest in this paper is to study the
Koszul property (in the Priddy's sense)   for graded skew PBW extensions. In the remainder of this
paper, $\mathbb{K}$ is a field and all algebras are $\mathbb{K}$-algebras.

\section{Graded skew PBW extensions}

Let $\mathbb{K}$ be a field. It is said that a $\mathbb{K}$-algebra $A$ is \emph{finitely graded}
(see \cite{Rogalski}) if the following conditions hold:
\begin{enumerate}
\item[\rm (i)] $A$ is $\mathbb{N}$-graded (positively graded): $A = \bigoplus_{j\geq 0}A_j$,
\item[\rm (ii)] $A$ is \emph{connected}, i.e., $A_0 = \mathbb{K}$,
\item[\rm (iii)] $A$ is \emph{finitely generated} as $\mathbb{K}$-algebra, i.e., there is a finite set of
elements $x_1,\dots, x_n\in A$ such that the set $\{x_{i_1}x_{i_2}\cdots x_{i_m}\mid 1\leq i_j\leq n, m\geq 1\} \cup \{1\}$ spans $A$ as
a $\mathbb{K}$-space.
\end{enumerate}

The algebra $A$ is \emph{augmented}, i.e., there is a canonical surjective $\mathbb{K}$-algebra
homomorphism $\varepsilon : A\twoheadrightarrow \mathbb{K}$; $\ker (\varepsilon)$ is called
\emph{augmentation ideal}. $A$ is called \emph{locally finite} if $dim_\mathbb{K}A_j<\infty$, for
all $j\in \mathbb{N}$. A graded $A$-module $M=\bigoplus_{j\in \mathbb{Z}}M_j$ is called
\emph{locally finite} if $dim_\mathbb{K}M_j<\infty$, for all $j\in \mathbb{Z}$. We say that the
graded A-module $M$ is \emph{generated in degree} $s$ if $M = A\cdot M_s$. $M$ is
\emph{concentrated} in degree $m$ if $M = M_{m}$. For any integer $l$, $M(l)$ is a graded
$A-$module whose degree $i$ component is $M(l)_i = M_{i+l}$.

The free associative algebra (tensor algebra) $L$ in $n$ generators $x_1,\dots, x_n$ is the ring $L:=\mathbb{K}\langle x_1,\dots, x_n\rangle$, whose
underlying $\mathbb{K}$-vector space  is the set of all words in the variables $x_i$, that is, expressions
$x_{i_1}x_{i_2}\dots\ x_{i_m}$ for some $m\geq 1$, where $1\leq i_j \leq n$ for all $j$. The length of a word $x_{i_1}x_{i_2}\dots x_{i_m}$ is
$m$. We include among the words a symbol 1, which we think of as the empty word, and which has
length 0. The product of two words is concatenation, and this operation is extended linearly to
define an associative product on all elements. 
Note that $L$ is positively graded with graduation given by $L:=\bigoplus_{j\geq 0}L_j$, where
$L_0= \mathbb{K}$ and $L_j$ spanned by all words of length $j$ in the alphabet $\{x_1, \dots,
x_n\}$, for $j>0$; $L$ is connected, the augmentation of $L$ is given by the natural projection
$\varepsilon: \mathbb{K}\langle x_1,\dots, x_n\rangle\to L_0= \mathbb{K}$ and the augmentation
ideal is given by $L_+:=\bigoplus_{j>0}L_j$. Let $P$ be a subspace of $F_2(L):=\mathbb{K}\bigoplus
L_1\bigoplus L_2$, the algebra $L/\langle P\rangle$ is called (nonhomogeneous) \emph{quadratic
algebra}. $L/\langle P\rangle$ is called \emph{homogeneous quadratic algebra}
if $P$ is a subspace of $L_2$, where $\langle P\rangle$ the two-sided ideal of $L$ generated by $P$.\\

\begin{proposition}\label{prop.Lez fga}
 Let $A$ be a connected $\mathbb{N}$-graded $\mathbb{K}$-algebra. $A$ is finitely generated
as $\mathbb{K}$-algebra if and only if $A = \mathbb{K}\langle x_1,\dots, x_m\rangle/I$, where $I$ is a proper homogeneous
two-sided ideal of $\mathbb{K}\langle x_1,\dots, x_m\rangle$. Moreover, for every $n\in \mathbb{N}$, dim$_\mathbb{K}A_n < \infty$, i.e., $A$ is locally finite.
\end{proposition}
\begin{proof}
$\Leftarrow)$: As the free algebra $L:= \mathbb{K}\langle x_1,\dots, x_m\rangle$ is
$\mathbb{N}$-graded and $I$ is homogeneous, i.e., graded, then $L/I$ es $\mathbb{N}$-graded with
graduation given by $(L/I)_n := (L_n + I)/I$. Note that $L/I$ is connected since $(L/I)_0 =
\mathbb{K}$. Moreover, $L/I$ is finitely generated as $\mathbb{K}$-algebra by the elements $x_i :=
x_i + I$, $1 \leq i \leq m$. Observe that $L_n$ is finitely generated as $\mathbb{K}$-vector space,
whence, $(L/I)_n$ is also
finitely generated as $\mathbb{K}$-vector space, i.e., dim$_\mathbb{K}((L/I)_n) < \infty$.\\
$\Rightarrow)$: Let $a_1, \dots, a_m \in A$ be a finite collection of elements that generate $A$ as
$\mathbb{K}$-algebra; by the universal property of the free algebra $\mathbb{K}\langle x_1,\dots, x_m\rangle$, there exists
a $\mathbb{K}$-algebra homomorphism $f : \mathbb{K}\langle x_1,\dots, x_m\rangle \to A$ with $f(x_i) := a_i$, $1 \leq i \leq m$; it
is clear that $f$ is surjective. Let $I := ker(f)$, then $I$ is a proper two-sided ideal of
$\mathbb{K}\langle x_1,\dots, x_m\rangle$ and
\begin{equation}\label{eq.iso A}
A \cong \mathbb{K}\langle x_1,\dots, x_m\rangle/I.
\end{equation}
Since $A$ is $\mathbb{N}$-graded, we can assume that every $a_i$ is homogeneous, $a_i\in A_{d_i}$ for
some $d_i\geq  1$, moreover, at least one of generators is of degree 1. We define a new
graduation for $L = \mathbb{K}\langle x_1,\dots, x_m\rangle$: we put weights $d_i$ to the variables $x_i$ and we set
$L'_n:={_\mathbb{K}}\langle x_{i_1}\cdots x_{i_m}\mid \sum_{j=1}^m d_{i_j} = n\rangle$ (the $\mathbb{K}$-space generated by $\{x_{i_1}\cdots x_{i_m}\mid \sum_{j=1}^m d_{i_j} = n\}$), $n \in \mathbb{N}$.
This implies that $f$ is graded, and from this we obtain that $I$ is homogeneous. In
fact, let $X_1 +\cdots + X_t \in I$, where $X_l\in  L'_{n_l}$, $1 \leq l \leq t$, so $f(X_1) +\cdots + f(X_t) = 0$,
and hence, $f(X_l) = 0$ for every $l$, i.e., $X_l \in I$.
Finally, note that under the isomorphism $\widetilde{f}$ in (\ref{eq.iso A}),  $\widetilde{f}((L'_n + I)/I) = A_n$, so
dim$_\mathbb{K}(A_n) < \infty$.
\end{proof}

Let $A$ be a finitely graded algebra; it is said that $A$ is \emph{finitely presented} if the two-sided ideal $I$ of relations in Proposition \ref{prop.Lez fga} is finitely generated.\\

We now  recall the definition of skew PBW extension  and some subclasses introduced  in
\cite{LezamaGallego}, \cite{LezamaAcostaReyes2015} and \cite{SuarezReyes2016}. We present
also some key properties of these extensions.

\begin{definition}\label{def.skewpbwextensions}
Let $R$ and $A$ be rings. We say that $A$ is a \textit{skew PBW
extension of} $R$ (also called a $\sigma$-PBW extension
of $R$) if the following conditions hold:
\begin{enumerate}
\item[\rm (i)]$R\subseteq A$;
\item[\rm (ii)]there exist finitely many elements $x_1,\dots ,x_n\in A$ such that $A$ is a left free $R$-module, with basis the basic elements
\begin{center}
${\rm Mon}(A):= \{x^{\alpha}=x_1^{\alpha_1}\cdots
x_n^{\alpha_n}\mid \alpha=(\alpha_1,\dots ,\alpha_n)\in
\mathbb{N}^n\}$.
\end{center}
\item[\rm (iii)]For each $1\leq i\leq n$ and any $r\in R\ \backslash\ \{0\}$, there exists an element $c_{i,r}\in R\ \backslash\ \{0\}$ such that
\begin{equation}\label{sigmadefinicion1}
x_ir-c_{i,r}x_i\in R.
\end{equation}
\item[\rm (iv)]For any elements $1\leq i,j\leq n$ there exists $c_{i,j}\in R\ \backslash\ \{0\}$ such that
\begin{equation}\label{sigmadefinicion2}
x_jx_i-c_{i,j}x_ix_j\in R+Rx_1+\cdots +Rx_n.
\end{equation}
Under these conditions we will write $A:=\sigma(R)\langle
x_1,\dots,x_n\rangle$.
\end{enumerate}
\end{definition}

The notation $\sigma(R)\langle x_1,\dots,x_n\rangle$ and the name of the skew PBW extensions is due to the following proposition.

\begin{proposition}[\cite{LezamaGallego}, Proposition 3]\label{sigmadefinition}
Let $A$ be a skew PBW extension of $R$. For each $1\leq i\leq
n$, there exists an injective endomorphism $\sigma_i:R\rightarrow
R$ and a $\sigma_i$-derivation $\delta_i:R\rightarrow R$ such that
\begin{equation}
x_ir=\sigma_i(r)x_i+\delta_i(r),\ \ \ \  \ r \in R.
\end{equation}
\end{proposition}

\begin{remark}\label{rem.grad f y rep} Let $A=\sigma(R)\langle x_1,\dots, x_n\rangle$ be a skew PBW extension with endomorphisms $\sigma_i$, $1\leq i \leq n$,
as in the Proposition \ref{sigmadefinition}. We establish the following notation (see \cite{LezamaGallego}, Definition 6).  $\alpha: = (\alpha_1,\dots, \alpha_n)\in \mathbb{N}^n$; $\sigma^{\alpha}:= (\sigma_1^{\alpha_1}\cdots  \sigma_n^{\alpha_n})$; $|\alpha|:=\alpha_1+\cdots+\alpha_n$; if $\beta: = (\beta_1,\dots, \beta_n)\in \mathbb{N}^n$, then $\alpha +\beta := (\alpha_1+ \beta_1,\dots, \alpha_n +\beta_n)$; for $X = x^{\alpha}=x_1^{\alpha_1}\cdots
x_n^{\alpha_n}$, exp($X$)$:= \alpha$ and deg($X$)$:= |\alpha|$. We have the following properties whose proof can be found in \cite{LezamaGallego}, Remark 2 and  Theorem 7.
\begin{enumerate}
\item[\rm (i)]Each element $f\in A\ \backslash\ \{0\}$ has a unique representation as $f=c_1X_1+\cdots+c_tX_t$, with $c_i\in R\ \backslash\ \{0\}$ and $X_i\in {\rm Mon}(A)$ for $1\leq i\leq t$.
\item[\rm (ii)] For every $x^{\alpha}\in {\rm Mon}(A)$ and every $0\neq r\in  R$, there exists unique elements $r_{\alpha}:=\sigma^{\alpha}(r)\in R \ \backslash\ \{0\}$ and $p_{\alpha, r}\in A$ such that $x^{\alpha}r = r_{\alpha}x^{\alpha} + p_{\alpha, r}$, where $p_{\alpha, r}= 0$ or deg$(p_{\alpha, r}) < |\alpha|$ if $p_{\alpha, r}\neq 0$.
\item[\rm (iii)] For every $x^{\alpha}$, $x^{\beta}\in {\rm Mon}(A)$ there exist unique elements $c_{\alpha,\beta}\in R$ and $p_{\alpha,\beta}\in A$ such
that $x^{\alpha}x^{\beta}= c_{\alpha,\beta}x^{\alpha+\beta} + p_{\alpha,\beta}$ where $c_{\alpha,\beta}$ is left invertible, $p_{\alpha,\beta}= 0$ or deg$(p_{\alpha,\beta}) <|\alpha + \beta|$ if $p_{\alpha,\beta}\neq 0$.
\end{enumerate}
\end{remark}

\begin{definition}\label{sigmapbwderivationtype}
Let $A$ be a skew PBW extension of $R$, $\Sigma:=\{\sigma_1,\dotsc, \sigma_n\}$ and $\Delta:=\{\delta_1,\dotsc, \delta_n\}$, where $\sigma_i$ and $\delta_i$ ($1\leq i\leq n$) are as in the Proposition \ref{sigmadefinition}
\begin{enumerate}
\item[\rm (a)]  $A$ is called {\em pre-commutative} if the conditions  {\rm(}iv{\rm)} in Definition
\ref{def.skewpbwextensions} are replaced by:\\
 For any $1\leq i,j\leq n$ there exists $c_{i,j}\in R\ \backslash\ \{0\}$ such that
\begin{equation}\label{relat.pre-comm}
x_jx_i-c_{i,j}x_ix_j\in Rx_1+\cdots +Rx_n.
\end{equation}

\item[\rm (b)]\label{def.quasicom} $A$ is called \textit{quasi-commutative} if the conditions
{\rm(}iii{\rm)} and {\rm(}iv{\rm)} in Definition
\ref{def.skewpbwextensions} are replaced by \begin{enumerate}
\item[\rm (iii')] for each $1\leq i\leq n$ and all $r\in R\ \backslash\ \{0\}$ there exists $c_{i,r}\in R\ \backslash\ \{0\}$ such that
\begin{equation}
x_ir=c_{i,r}x_i;
\end{equation}
\item[\rm (iv')]for any $1\leq i,j\leq n$ there exists $c_{i,j}\in R\ \backslash\ \{0\}$ such that
\begin{equation}
x_jx_i=c_{i,j}x_ix_j.
\end{equation}
\end{enumerate}
\item[\rm (c)]  $A$ is called \textit{bijective} if $\sigma_i$ is bijective for each $\sigma_i\in \Sigma$, and $c_{i,j}$ is invertible for any $1\leq
i<j\leq n$.
\item[\rm (d)] If $\sigma_i={\rm id}_R$ for every $\sigma_i\in \Sigma$, we say that $A$ is a skew PBW extension of {\em derivation type}.
\item[\rm (e)]  If $\delta_i = 0$ for every $\delta_i\in \Delta$, we say that $A$ is a skew PBW extension of {\em endomorphism type}.
\item[\rm (f)]   Any element $r$ of $R$ such that $\sigma_i(r)=r$ and $\delta_i(r)=0$ for all $1\leq i\leq n$ will be called a {\em constant}. $A$ is called \emph{constant} if every element of $R$ is constant.
\item[\rm (g)]  $A$ is called {\em semi-commutative} if $A$ is quasi-commutative and constant.
\end{enumerate}
\end{definition}

Let $I\subseteq \sum_{n\geq 2} L_{n}$ be a finitely-generated homogeneous ideal of $\mathbb{K}\langle x_1,\dots, x_n\rangle$ and let $R = \mathbb{K}\langle x_1,\dots, x_n\rangle/I$, which is a connected-graded $\mathbb{K}$-algebra generated in degree 1. Suppose $\sigma : R \to R$ is a graded algebra automorphism and $\delta : R(-1) \to R$ is a graded $\sigma$-derivation (i.e. a degree +1 graded $\sigma$-derivation $\delta$ of $R$). Let  $A := R[x; \sigma,\delta]$  be the associated \emph{graded Ore extension} of $R$; that is, $A = \bigoplus_{n\geq 0} Rx^n$ as an $R$-module, and for $r\in R$, $xr = \sigma(r)x + \delta(r)$. We consider $x$ to have degree 1 in $A$, and under this grading $A$ is a connected graded algebra generated in degree 1 (see \cite{Cassidy2008} and \cite{Phan}). We introduce the definition of graded skew PBW extensions following \cite{Cassidy2008}.

\begin{definition}\label{def. graded skew PBW ext} Let  $A=\sigma(R)\langle x_1,\dots, x_n\rangle$ be a bijective skew PBW extension of a  $\mathbb{N}$-graded $\mathbb{K}$-algebra $R$. We said that $A$ is a \emph{graded  skew PBW extension} if the following conditions hold:

\begin{enumerate}
\item[\rm (i)] $x_1,\dots, x_n$ have degree 1 in $A$.
\item[\rm (ii)] $\sigma_i$ is a graded ring homomorphism and $\delta_i : R(-1) \to R$ is a graded $\sigma_i$-derivation for all $1\leq i  \leq n$, where $\sigma_i$ and $\delta_i$ are as in the Proposition \ref{sigmadefinition}.
\item[\rm (iii)]  $x_jx_i-c_{i,j}x_ix_j\in R_2+R_1x_1 +\cdots + R_1x_n$, as in (\ref{sigmadefinicion2}) and $c_{i,j}\in R_0$.
\end{enumerate}
 \end{definition}

\begin{proposition}\label{prop.grad A}
Let $A$ be a graded skew PBW extension of $R$ and let $A_p$ the $\mathbb{K}$-space generated by the set
\[\Bigl\{r_tx^{\alpha} \mid t+|\alpha|= p,\  r_t\in R_t \text{  and } x^{\alpha}\in {\rm Mon}(A)\Bigr\},
\]
for $p\geq 0$. Then:
\begin{enumerate}
\item[\rm(i)]  $R_p\subseteq A_p$ for each $p\geq 0$.
\item[\rm(ii)] $A$ is a  graded $\mathbb{K}$-algebra with graduation
\begin{equation}\label{eq.grad alg skew}
A=\bigoplus_{p\geq 0} A_p.
\end{equation}
\item[\rm(iii)] $A$ is a graded  $R$-module  with the above graduation.
\end{enumerate}
\end{proposition}
\begin{proof}
Let $R=\bigoplus_{p\geq 0} R_p$ be a graded algebra and let $A=\sigma(R)\langle x_1, \dots, x_n\rangle$ be a graded skew PBW extension.\\
\rm(i): If $r_p\in R_p$, then $r_p= r_px^0_1\cdots x^0_n  \in A_p$.\\
\rm(ii): It is clear that $1=x^0_1\cdots x^0_n\in A_0$. Let $f\in A\ \backslash\ \{0\}$, then by
Remark \ref{rem.grad f y rep}, $f$  has a unique representation as $f=r_1X_1+\cdots+r_sX_s$, with
$r_i\in R\ \backslash\ \{0\}$ and $X_i:=x_1^{\alpha_{i_1}}\cdots x_n^{\alpha_{i_n}}\in {\rm
Mon}(A)$ for $1\leq i\leq s$. Let $r_i= r_{i_{q_1}}+ \cdots + r_{i_{q_m}}$ the unique
representation of $r_i$ in homogeneous elements of $R$. Then $f= (r_{1_{q_1}}+ \cdots +
r_{1_{q_m}}) x_1^{\alpha_{1_1}}\cdots x_n^{\alpha_{1_n}}+\cdots + (r_{s_{q_1}}+ \cdots +
r_{s_{q_u}})x_1^{\alpha_{s_1}}\cdots x_n^{\alpha_{s_n}}= r_{1_{q_1}} x_1^{\alpha_{1_1}}\cdots
x_n^{\alpha_{1_n}}+ \cdots + r_{1_{q_m}} x_1^{\alpha_{1_1}}\cdots x_n^{\alpha_{1_n}}+ \cdots +
r_{s_{q_1}}x_1^{\alpha_{s_1}}\cdots x_n^{\alpha_{s_n}} + \cdots +
r_{s_{q_u}}x_1^{\alpha_{s_1}}\cdots x_n^{\alpha_{s_n}}$ is  the unique representation of $f$ in
homogeneous elements of $A$. Therefore $A$ is a direct sum of the family $\{A_p\}_{p\geq 0}$ of
subspaces of $A$.

Now, let $x\in A_pA_q$. Without loss of generality we can assume that  $x= (r_tx^{\alpha})(r_sx^{\beta})$ with $r_t\in R_t$,  $r_s\in R_s$, $ x^{\alpha},  x^{\beta}\in {\rm Mon}(A)$, $t+|\alpha|= p$ and $s+|\beta|= q$. By Remark \ref{rem.grad f y rep}-{\rm(ii)}, we have that for $r_s$ and $x^{\alpha}$ there exists unique elements $r_{s_{\alpha}}:=\sigma^{\alpha}(r_s)\in R \ \backslash\ \{0\}$ and $p_{\alpha, r_s}\in A$ such
that $x= r_t(r_{s_{\alpha}}x^{\alpha} + p_{\alpha, r_s})x^{\beta}= r_tr_{s_{\alpha}}x^{\alpha}x^{\beta}+ r_tp_{\alpha, r_s}x^{\beta}$,
where $p_{\alpha, r_s}= 0$ or deg$(p_{\alpha, r_s}) < |\alpha|$ if $p_{\alpha, r_s}\neq 0$. Now, by Remark \ref{rem.grad f y rep}-{\rm(iii)},
we have that for $x^{\alpha}$, $x^{\beta}$ there exists unique elements $c_{\alpha,\beta}\in R$ and $p_{\alpha,\beta}\in A$ such
that $x=r_tr_{s_{\alpha}}(c_{\alpha,\beta}x^{\alpha+\beta} + p_{\alpha,\beta})+ r_tp_{\alpha, r_s}x^{\beta}=r_tr_{s_{\alpha}}c_{\alpha,\beta}x^{\alpha+\beta} +
r_tr_{s_{\alpha}}p_{\alpha,\beta}+ r_tp_{\alpha, r_s}x^{\beta}$, where $c_{\alpha,\beta}$ is left invertible,
$p_{\alpha,\beta}= 0$ or deg$(p_{\alpha,\beta}) <|\alpha + \beta|$ if $p_{\alpha,\beta}\neq 0$. We note that:
\begin{enumerate}
\item Since $\sigma_i$ is graded  for $1\leq i\leq n$, then $\sigma_i^{\alpha_i}$ is graded and therefore $\sigma^{\alpha}$ is graded.
Then $r_{s_{\alpha}}:=\sigma^{\alpha}(r_s)\in R_s$ and  $\delta_i^{\alpha_i}(r_s)\in R_{s+\alpha_i}$, for $1\leq i\leq n$
and $\alpha_i\geq 0$.
\item $x_i^{\alpha_i}r_s=\sigma_i^{\alpha_i}(r_s)x_i^{\alpha_i}+\delta_i(\sigma_i^{\alpha_i-1}(r_s))x_i^{\alpha_i-1}+ \delta_i^2(\sigma_i^{\alpha_i-2}(r_s))x_i^{\alpha_i-2}+\cdots + \delta_i^j(\sigma_i^{\alpha_i-j}(r_s))x_i^{\alpha_i-j}+\cdots + \delta_i^{\alpha_i-1}(\sigma_i(r_s))x_i + \delta_i^{\alpha_i}(r_s)\in A_{s+\alpha_i}$, since each summand in the above expression is in $A_{s+\alpha_i}$.
\item From the Definition \ref{def. graded skew PBW ext}, we have that for $1\leq i< j\leq n$, $x_jx_i=c_{i,j}x_ix_j+ r_{0_{ij}} + r_{1_{ij}}x_1
+ \cdots + r_{n_{ij}}x_n\in A_2$. Then, for  $1\leq i< j< k\leq n$, we have that
    \begin{align*}
    x_k(x_jx_i)  = & x_k(c_{i,j}x_ix_j+ r_{0_{ij}} + r_{1_{ij}}x_1 + \cdots + r_{n_{ij}}x_n)\\
     = & (\sigma_k(c_{i,j})x_kx_ix_j+\delta_k(c_{i,j})x_ix_j) + (\sigma_k(r_{0_{ij}})x_k +\delta_k(r_{0_{ij}}))\\
     & + (\sigma_k(r_{1_{ij}})x_kx_1 +\delta_k(r_{1_{ij}})x_1)+\cdots
    + (\sigma_k(r_{n_{ij}})x_kx_n +\delta_k(r_{n_{ij}})x_n)\\
     = & \sigma_k(c_{i,j})[c_{i,k}x_ix_k+ r_{0_{ik}} + r_{1_{ik}}x_1 + \cdots + r_{n_{ik}}x_n]x_j +\delta_k(c_{i,j})x_ix_j + \sigma_k(r_{0_{ij}})x_k\\
      & +\delta_k(r_{0_{ij}}) + \sigma_k(r_{1_{ij}})[c_{1,k}x_1x_k + r_{0_{1k}} + r_{1_{1k}}x_1 + \cdots + r_{n_{1k}}x_n ]\\
      & +\delta_k(r_{1_{ij}})x_1+\cdots + \sigma_k(r_{n_{ij}})x_kx_n+\delta_k(r_{n_{ij}})x_n\\
    = & \sigma_k(c_{i,j})c_{i,k}x_i[c_{j,k}x_jx_k + r_{0_{ij}} + r_{1_{jk}}x_1 + \cdots + r_{n_{jk}}x_n] + \sigma_k(c_{i,j}) r_{0_{ik}}x_j\\
    & + \sigma_k(c_{i,j}) r_{1_{ik}}x_1x_j + \cdots  + \sigma_k(c_{i,j}) r_{n_{ik}}x_nx_j  +\delta_k(c_{i,j})x_ix_j + \sigma_k(r_{0_{ij}})x_k\\
    & +\delta_k(r_{0_{ij}}) + \sigma_k(r_{1_{ij}})[c_{1,k}x_1x_k  + r_{0_{1k}} + r_{1_{1k}}x_1 + \cdots   + r_{n_{1k}}x_n ]\\
     & +\delta_k(r_{1_{ij}})x_1+\cdots      + \sigma_k(r_{n_{ij}})x_kx_n+\delta_k(r_{n_{ij}})x_n\\
     = & \sigma_k(c_{i,j})c_{i,k}\sigma(c_{j,k})x_ix_jx_k+ \sigma_k(c_{i,j})c_{i,k}\delta_i(c_{j,k})x_jx_k + \sigma_k(c_{i,j})c_{i,k}\sigma(c_{j,k})\sigma_i(r_{0_{ij}})x_i \\
    &   + \sigma_k(c_{i,j})c_{i,k}\delta_i(r_{0_{ij}}) + \sigma_k(c_{i,j})c_{i,k}\sigma_i(r_{1_{jk}})x_ix_1 +   \sigma_k(c_{i,j})c_{i,k}\delta_i(r_{1_{jk}})x_1+ \cdots \\
     &   + \sigma_k(c_{i,j})c_{i,k}\sigma_i(r_{n_{jk}})x_ix_n + \sigma_k(c_{i,j})c_{i,k}\delta_i(r_{n_{jk}})x_n + \sigma_k(c_{i,j}) r_{0_{ik}}x_j \\
     & +  \sigma_k(c_{i,j}) r_{1_{ik}}x_1x_j + \cdots + \sigma_k(c_{i,j}) r_{n_{ik}}x_nx_j +\delta_k(c_{i,j})x_ix_j + \sigma_k(r_{0_{ij}})x_k  +\delta_k(r_{0_{ij}}) \\
     & + \sigma_k(r_{1_{ij}})c_{1,k}x_1x_k + \sigma_k(r_{1_{ij}})r_{0_{1k}} +\sigma_k(r_{1_{ij}}) r_{1_{1k}}x_1 + \cdots + \sigma_k(r_{1_{ij}})r_{n_{1k}}x_n  \\
     & +\delta_k(r_{1_{ij}})x_1+\cdots + \sigma_k(r_{n_{ij}})x_kx_n+\delta_k(r_{n_{ij}})x_n\\
     = & \sigma_k(c_{i,j})c_{i,k}\sigma(c_{j,k})x_ix_jx_k+ \sigma_k(c_{i,j})c_{i,k}\delta_i(c_{j,k})x_jx_k + \sigma_k(c_{i,j})c_{i,k}\sigma(c_{j,k})\sigma_i(r_{0_{ij}})x_i \\
    &   + \sigma_k(c_{i,j})c_{i,k}\delta_i(r_{0_{ij}}) + \sigma_k(c_{i,j})c_{i,k}\sigma_i(r_{1_{jk}})c_{1,i}x_1x_i + \sigma_k(c_{i,j})c_{i,k}\sigma_i(r_{1_{jk}})r_{0_{1i}}\\
     & + \sigma_k(c_{i,j})c_{i,k}\sigma_i(r_{1_{jk}})r_{1_{1i}}x_1 + \cdots   + \sigma_k(c_{i,j})c_{i,k}\sigma_i(r_{1_{jk}}) r_{n_{1i}}x_n \\
     &   +    \sigma_k(c_{i,j})c_{i,k}\delta_i(r_{1_{jk}})x_1+ \cdots   + \sigma_k(c_{i,j})c_{i,k}\sigma_i(r_{n_{jk}})x_ix_n + \sigma_k(c_{i,j})c_{i,k}\delta_i(r_{n_{jk}})x_n\\
     &  + \sigma_k(c_{i,j}) r_{0_{ik}}x_j  +  \sigma_k(c_{i,j}) r_{1_{ik}}x_1x_j + \cdots +  \sigma_k(c_{i,j}) r_{n_{ik}}c_{j,n}x_jx_n +
      \sigma_k(c_{i,j}) r_{n_{ik}} \\
      & + \sigma_k(c_{i,j}) r_{n_{ik}}r_{0_{jn}}+\sigma_k(c_{i,j}) r_{n_{ik}}r_{1_{jn}}x_1+\cdots + \sigma_k(c_{i,j}) r_{n_{ik}}r_{n_{jn}}x_n + \delta_k(c_{i,j})x_ix_j\\
      & + \sigma_k(r_{0_{ij}})x_k  +\delta_k(r_{0_{ij}})
     + \sigma_k(r_{1_{ij}})c_{1,k}x_1x_k + \sigma_k(r_{1_{ij}})r_{0_{1k}} +\sigma_k(r_{1_{ij}}) r_{1_{1k}}x_1 + \cdots\\
      & + \sigma_k(r_{1_{ij}})r_{n_{1k}}x_n
      + \delta_k(r_{1_{ij}})x_1+ \cdots + \sigma_k(r_{n_{ij}})x_kx_n+\delta_k(r_{n_{ij}})x_n.
     \end{align*}
Since all summands in the above sum have the form $rx$, where $r$ is an homogeneous element of $R$,
$x\in {\rm Mon}(A)$ and $rx\in A_3$, we have that  $x_kx_jx_i\in A_3$. Following this procedure we
get in general that $x_{i_1}x_{i_2}\cdots x_{i_m}\in A_m$ for $1\leq i_k\leq n$, $1\leq k \leq m$,
$m\geq 1$.
\item In a similar way and following the proof of \cite{LezamaGallego}, Theorem 7, we obtain that
$x^{\alpha}r_s\in A_{|\alpha| +s}$, and since $c_{\alpha,\beta}\in R_0$, then
$x^{\alpha}x^{\beta}\in A_{|\alpha|+|\beta|}$. Therefore $p_{\alpha, r_s}\in  A_{|\alpha| +s}$ and
$p_{\alpha,\beta}\in A_{|\alpha|+|\beta|}$. Then
$r_tr_{s_{\alpha}}c_{\alpha,\beta}x^{\alpha+\beta}\in A_{t+s+|\alpha|+|\beta|}$,
$r_tr_{s_{\alpha}}p_{\alpha,\beta}\in A_{t+s+|\alpha|+|\beta|}$ and $r_tp_{\alpha, r_s}x^{\beta}\in
A_{t+|\alpha|+ s+|\beta|}$, i.e., $x\in A_{p+q}$.
\end{enumerate}
\rm(iii): This follows from \rm(ii).
\end{proof}

\begin{example} Quasi-commutative skew PBW extensions with  the trivial graduation of  $R$ is a graded   skew PBW extensions:
Let $r\in R=R_0$, then $\sigma_i(r)=c_{i,r}\in R_0$, $\delta_i=0$ and  $x_jx_i-c_{i,j}x_ix_j=0 \in R_2+R_1x_1 +\cdots + R_1x_n$;
if we assume that  $R$  has a different graduation to the trivial graduation, then $A$ is graded skew PBW extension  provided that $\sigma_i$
is graded and $c_{i,j}\in R_0$,  $1\leq i,j \leq n$.
\end{example}

\begin{examples}\label{ex.grad skew pbw ext} Next we present specific examples of graded skew PBW extensions  of the classical polynomial ring $R$
with coefficients in a field $\mathbb{K}$, which are not quasi-commutative and where $R$ has the
usual graduation. In \cite{LezamaGallego}, \cite{LezamaReyes} and \cite{SuarezReyes2016} we
can be found further details of these algebras.
\begin{enumerate}
\item The Jordan plane. $A=\mathbb{K}\langle x, y\rangle/ \langle yx-xy-x^2\rangle\cong \sigma(\mathbb{K}[x])\langle y\rangle$.

\item The homogenized enveloping algebra. $\mathcal{A}(\mathcal{G})\cong \sigma(\mathbb{K}[z])\langle x_1,\dots, x_n\rangle$.

\item The Diffusion algebra 2.  $A \cong \sigma(\mathbb{K}[x_1, \dots, x_n])\langle D_1,\dots, D_n\rangle$.

\item The algebra $U\cong\sigma(\mathbb{K}[x_1,\dotsc,x_n]) \langle
y_1,\dotsc,y_n;z_1,\dotsc,z_n\rangle$.

\item Manin algebra. $\cO(M_q(2))\cong \sigma(\mathbb{K}[u])\langle x, y, v\rangle$.

\item  Algebra of quantum matrices. $\cO_q(M_n(\mathbb{K}))\cong
\sigma(\mathbb{K}[x_{im}, x_{jk}])\langle x_{ik},
x_{jm}\rangle$, for $1\le i<j,k<m\le n$.

\item Quadratic algebras. If $a_1 = a_4 = 0$ then the quadratic algebra is  a graded   skew PBW extension of $R= \mathbb{K}[y, z]$, and if $a_5 = a_3 = 0$ then quadratic algebras are graded skew PBW extensions of $R= \mathbb{K}[x, z]$.
\end{enumerate}

\end{examples}

\begin{remark}\label{rem.prop of graded skew} Let $A=\sigma(R)\langle x_1,\dots, x_n\rangle$ be a graded skew $PBW$ extension. Then  we immediately have the following properties:
\begin{enumerate}
\item[\rm (i)] $A$ is a $\mathbb{N}$-graded $\mathbb{K}$-algebra and  $A_0=R_0$.
\item[\rm (ii)] $R$ is connected if and only if $A$ is connected.
\item[\rm (iii)] If $R$ is finitely generated then $A$ is finitely generated. Indeed, as ${\rm Mon}(A)= \{x^{\alpha}=x_1^{\alpha_1}\cdots
x_n^{\alpha_n}\mid \alpha=(\alpha_1,\dots ,\alpha_n)\in
\mathbb{N}^n\}$  is $R$-base for $A$, and  $R$ is finitely generated as $\mathbb{K}$-algebra, then there is a finite set of elements $t_1,\dots, t_s\in R$ such that the set $\{t_{i_1}t_{i_2}\cdots t_{i_m}|1\leq i_j\leq s, m\geq 1\} \cup \{1\}$ spans $R$ as
a $\mathbb{K}$-space. Then there is a finite set of elements $t_1,\dots, t_s, x_1, \dots, x_n\in A$ such that the set $\{t_{i_1}t_{i_2}\cdots t_{i_m}x_1^{\alpha_1}\cdots
x_n^{\alpha_n}\mid 1\leq i_j\leq s, m\geq 1,\alpha_1,\dots ,\alpha_n\in
\mathbb{N}\}$ spans $A$ as a $\mathbb{K}$-space. So, if $R$ is generated in degree 1 then $A$ is generated in degree 1.
\item[\rm(iv)] For {\rm (i)}, {\rm (ii)} and {\rm (iii)} above we have that if $R$ is a finitely graded algebra then $A$ is a finitely graded algebra.
\item[\rm (v)] If $R$ is locally finite, $A$ as $\mathbb{K}$-algebra is a locally finite. Indeed, $dim_\mathbb{K}A_0 = dim_\mathbb{K}R_0$, $dim_\mathbb{K}A_1 = dim_\mathbb{K}R_1 + n$; let $\mathcal{B}_t$ be a (finite) base of $R_t$, $t\geq 0$, then for a fixed $p\geq 2$ the set $\{r_tx^{\alpha} \mid t+|\alpha|= p,\  r_t\in B_t \text{  and } x^{\alpha}\in {\rm Mon}(A)\}$ is a finite base  for $A_p$.
\item[\rm (vi)] $A$ as $R$-module  is locally finite.

\item[\rm(vii)] If $A$ is quasi-commutative and $R$ is concentrate in degree 0, then $A_0=R$.
\item[\rm(viii)] If $R$ is a homogeneous quadratic algebra then $A$ is a homogeneous quadratic algebra.

\item[\rm(ix)] If $R$ is finitely presented then $A$ is finitely presented. Indeed: by Proposition \ref{prop.Lez fga},  $R= \mathbb{K}\langle t_1,\dots, t_m\rangle/I$ where
    \begin{equation}\label{rel1.R}
I=\langle r_1,\dots, r_s\rangle
\end{equation}
     is a two-sided ideal of $\mathbb{K}\langle t_1, \dots, t_m\rangle$ generated by a finite set $r_1,\dots, r_s$ of homogeneous polynomials in $\mathbb{K}\langle t_1, \dots, t_m\rangle$. Then $A= \mathbb{K}\langle t_1, \dots, t_m, x_{1}, \dots, x_{n}\rangle/J$ where
 \begin{equation}\label{eq1.ideal J}
 J=\langle r_1,\dots, r_s, \ f_{hk}, \ g_{ji} \mid 1\leq i,j, h\leq n,\  1\leq k\leq m\rangle
 \end{equation}
  is the two-sided ideal of $\mathbb{K}\langle t_1, \dots, t_m, x_{1}, \dots, x_{n}\rangle$ generated by a finite set of homogeneous elements $r_1,\dots, r_s$, $f_{hk}$, $g_{ji}$ where $r_1,\dots, r_s$ are as in  (\ref{rel1.R});
 \begin{equation}\label{eq1.rel xt}
 f_{hk}:= x_{h}t_k-\sigma_{h}(t_k)x_{h}-\delta_{h}(t_k)
 \end{equation}
 with $\sigma_{h}$ and $\delta_{h}$ as in Proposition \ref{sigmadefinition};
  \begin{equation}\label{eq1.rel xx}
 g_{ji}:= x_{j}x_{i}-c_{i,j}x_{i}x_{j} - (r_{0_{j,i}} + r_{1_{j,i}}x_{1} + \cdots + r_{n_{j,i}}x_{n})
 \end{equation}
 as in  (\ref{sigmadefinicion2}) of Definition \ref{def.skewpbwextensions}.

\end{enumerate}
\end{remark}

\begin{remark}\label{rem.grad skew no impl iterat}
The class of  graded iterated Ore extensions $\subsetneqq$ class of graded skew PBW extensions. For example, the homogenized enveloping algebra $\mathcal{A}(\mathcal{G})$ and the Diffusion algebra 2  are graded   skew PBW extension but this is not iterated Ore extensions. Therefore, the definition of graded skew PBW extensions is more general that the  definition of graded Ore extensions.
\end{remark}

\section{Koszul algebras}

Let $A=\mathbb{K}\bigoplus A_1\bigoplus A_2\bigoplus\cdots$ be a locally finite graded algebra  and
$E(A) =\bigoplus_{s,p}E^{s,p}(B) = \bigoplus_{s,p}Ext^{s,p}_A(\mathbb{K}, \mathbb{K})$ the
associated bigraded Yoneda algebra, where $s$ is the cohomology degree  and $-p $ is the internal
degree inherited from the grading on $A$. Let $E^{s}(A) = \bigoplus_{p}E^{s,p}(A)$. $A$  is called
\emph{Koszul} if the following equivalent conditions hold (see \cite{Polishchuk}, Chapter 2,
Definition 1):
\begin{enumerate}
\item[\rm (i)]  $Ext^{s,p}_A(\mathbb{K}, \mathbb{K})=0$ for $s\neq p$;
\item[\rm (ii)] $A$ is one-generated and the algebra  $Ext^*_A (\mathbb{K}, \mathbb{K})$ it is generated by  $Ext_A^{1}(\mathbb{K}, \mathbb{K})$, i.e., $E(A)$ is generated in the first cohomological degree;
\item[\rm (iii)] The module $\mathbb{K}$ admits a \emph{linear free resolution}, i.e., a resolution by free $A$-modules
\[
\cdots \to P_2\to P_1\to P_0\to \mathbb{K}\to 0
\]
such that $P_i$ is generated in degree $i$.
\end{enumerate}

Let $A$ be a graded Ore extension of $R$. Then $A$  is homogeneous quadratic if and only if $R$ is homogeneous quadratic. Furthermore,  $A$ is  Koszul if and  only if $R$ is  Koszul (see \cite{Phan},  Corollary 1.3).

\begin{proposition}\label{iterated Ore}
The graded iterated Ore extension $A := R[x_1; \sigma_1,\delta_1]\cdots [x_{n}; \sigma_{n},\delta_{n}]$ is Koszul if and only if $R$ is Koszul.
\end{proposition}

\begin{proof}
Suppose
\[
\sigma_i : R[x_1; \sigma_1,\delta_1]\cdots [x_{i-1}; \sigma_{i-1},\delta_{i-1}]  \to R[x_1; \sigma_1,\delta_1]\cdots [x_{i-1}; \sigma_{i-1},\delta_{i-1}]
\]
is a graded algebra automorphism and
\[
\delta_i : R[x_1; \sigma_1,\delta_1]\cdots [x_{i-1}; \sigma_{i-1},\delta_{i-1}](-1) \to R[x_1; \sigma_1,\delta_1]\cdots [x_{i-1}; \sigma_{i-1},\delta_{i-1}]
\]
 is a graded $\sigma_i$-derivation, $2\leq i\leq n$. Let  $A := R[x_1; \sigma_1,\delta_1]\cdots [x_{n}; \sigma_{n},\delta_{n}]$  be the \emph{graded iterated Ore extension} of $R$, where $x_1,\dots, x_n$ have degree 1 in $A$. Then from \cite{Phan}, Corollary 1.3 the  result is clear.

\end{proof}

\begin{proposition}[\cite{LezamaReyes}, Theorem 2.3]\label{1.3.3}
Let $A$ be a quasi-commutative skew PBW extension of a ring $R$.
Then  {\rm (i)} $A$ is isomorphic to an iterated skew polynomial ring, and {\rm (ii)} if $A$ is bijective,  each endomorphism of the skew polynomial ring in {\rm (i)} is an isomorphism.
\end{proposition}

\begin{proposition}\label{exten Koszul hom}
Let $A$ be a graded quasi-commutative skew PBW extension  of $R$. Then $R$ is a  Koszul algebra if and only if $A$  is Koszul.
\end{proposition}

\begin{proof}
If $A$ is a graded quasi-commutative bijective skew PBW extension of $R$, then by Proposition \ref{1.3.3} $A$  is isomorphic to an iterated graded Ore extension  wherein each endomorphism  is bijective. Then  by Proposition \ref{iterated Ore}, $R$ is Koszul if and only if $A$ is Koszul.
\end{proof}

\section{PBW algebras}

Let $L=\mathbb{K}\langle x_1,\dots, x_n\rangle$ and let $A = \mathbb{K}\langle x_1,\dots,
x_n\rangle/\langle P\rangle$ be a homogeneous quadratic algebra with a fixed generators
$\{x_1,\dots, x_n\}$. For a multindex $\alpha:=(i_1,\dots, i_m)$, where $1\leq i_k\leq n$, we
denote the monomials  $x^{\alpha}:=x_{i_1}x_{i_2}\cdots x_{i_m}\in \mathbb{K}\langle x_1,\dots,
x_n\rangle$. For $\alpha=\emptyset$ we set $x^{\emptyset}:=1$. Now let us equip the subspace $L_2$
with the basis consisting of the monomials $x_{i_1}x_{i_2}$. Let $S^{(1)}:=\{1, 2,\dots,n\}$,
$S^{(1)}\times S^{(1)}$ the cartesian product, then for $P\subseteq L_2$ we obtain the set
$S\subseteq S^{(1)}\times S^{(1)}$ of  pairs of indices $(l, m)$ for which the class of $x_lx_m$ in
$L_2/P$ is not in the span of the classes of $x_rx_s$ with $(r, s) < (l, m)$, where $<$ denotes the
lexicographical order. Hence, the relations in $A$ can be written in the following form (see
\cite{Polishchuk}, Lemma 4.1.1):

\[
x_ix_j =\sum_{\substack{(r,s)<(i,j)\\(r,s)\in S}}c^{rs}_{ij} x_rx_s,\qquad (i, j) \in S^{(1)}\times
S^{(1)}\ \backslash\ S.
\]
Define further $S^{(0)}:= \{\emptyset\}$,  and for $m\geq 2$,
\[
S^{(m)}:= \{(i_1,\dots, i_m) \mid (i_k, i_{k+1}) \in S, \ k = 1,\dots, m - 1\}
\]
and consider the monomials $\{x_{i_1}\cdots x_{i_m} \in A_m \mid  (i_1,\dots, i_m) \in S^{(m)}\}$. Note that these mo\-no\-mials always span $A_m$ as a vector space and  the monomials
\begin{equation}\label{eq.monPBW}
(A,S):=\{x_{i_1}\cdots x_{i_m} \mid  (i_1,\dots, i_m) \in \cup_{m> 0}S^{(m)}\}
\end{equation}
linearly span the entire $A$.  We call $(A,S)$ in (\ref{eq.monPBW}) a \emph{PBW-basis} of $A$ if they are linearly independent and hence form a $\mathbb{K}$-linear basis.  The elements $x_1, \dots, x_n$ are called PBW-\emph{generators} of $A$. A \emph{PBW-algebra} is a homogeneous quadratic algebra admitting a PBW-basis, i.e., there exists a permutation of $x_1,\dots, x_n$ such that the standard monomials in $x_1,\dots, x_n$ conform a $\mathbb{K}$-basis of $A$. In \cite{SuarezReyes2016} we  show that every semi-commutative skew PBW extension of $\mathbb{K}$ is a PBW algebra.

\begin{proposition}\label{prop.R PBW impl A PBW}
Let $A$ be a graded skew PBW extension of  a finitely presented algebra  $R$. If $R$ is a PBW
algebra then $A$ is a PBW algebra.
\end{proposition}
\begin{proof}

Let $R$ be a finitely presented PBW algebra with PBW generators $t_1,\dots, t_m$.
Then by Proposition \ref{prop.Lez fga}, $R= L^t/I$, where $L^t=\mathbb{K}\langle t_1, \dots,
t_m\rangle$ and
\begin{equation}\label{rel.R}
I=\langle r_1,\dots, r_s\rangle
\end{equation}
 is a two-sided ideal of $\mathbb{K}\langle t_1, \dots, t_m\rangle$ generated by a  finite   set $r_1,\dots, r_s$ of homogeneous polynomials in $\mathbb{K}\langle t_1, \dots, t_m\rangle$ of degree two. Let
 \begin{equation}\label{eq.mon R}
 (R,S_t):= \{t_{i_1}\cdots t_{i_v}\mid   (i_1,\dots, i_{\nu}) \in \cup_{p> 0}S_t^{(p)} \}
 \end{equation}
be a PBW basis of $R$, with $S_t^{(p)}= \{(i_1,i_2,\dots,i_p)\mid (i_k, i_{k+1}) \in S_t, \ k = 1,\dots, p - 1\}$, $S_t^{(1)}:= \{1,2,\dots, m\}$ and $S_t\subseteq S_t^{(1)}\times S_t^{(1)}$ is the set of  pairs of indices $(i_\mu, i_\nu)$ for which the class of $t_{i_\mu}t_{i_\nu}$ in $L_2^t/P$ (where $P$ is the space of relations $r_1,\dots, r_s$)  is not in the span of the classes of $t_rt_s$ with $(r, s) < (i_\mu, i_\nu)$. For $1\leq d\leq s$,
\begin{equation}\label{eq.rel tt}
r_d=t_{i_d}t_{j_d} =\sum_{\substack{(r_d,q_d)<(i_d,j_d)\\(r_d,q_d)\in S_t}} c^{r_dq_d}_{i_dj_d}  t_{r_d} t_{q_d},\qquad (i_d, j_d) \in S_t^{(1)}\times S_t^{(1)}\ \backslash\ S_t.
\end{equation}
 Let $A=\sigma(R)\langle x_{m+1}, \dots, x_{m+n}\rangle$ be a graded skew PBW extension of $R$. As $R\subseteq A$, we have that $A=\mathbb{K}\langle t_1, \dots, t_m, x_{m+1}, \dots, x_{m+n}\rangle/J$ where
 \begin{equation}\label{eq.ideal J}
 J=\langle r_1,\dots, r_s, \ f_{hk}, \ g_{ji} \mid m+1\leq i,j, h\leq m+n,\  1\leq k\leq m\rangle
 \end{equation}
 is the two-sided ideal of $\mathbb{K}\langle t_1, \dots, t_m, x_{m+1}, \dots, x_{m+n}\rangle$ generated
 by a  set $r_1,\dots, r_s, \ f_{hk}, \ g_{ji}$ where $r_1,\dots, r_s$ are as in (\ref{rel.R});
 let
 \begin{equation}\label{eq.rel xt}
 f_{hk}:= x_{m+h}t_k-\sigma_{m+h}(t_k)x_{m+h}-\delta_{m+h}(t_k)
 \end{equation}
 with $\sigma_{m+h}$ and $\delta_{m+h}$ as in Proposition \ref{sigmadefinition};
  \begin{equation}\label{eq.rel xx}
 g_{ji}:= x_{m+j}x_{m+i}-c_{i,j}x_{m+i}x_{m+j} - (r_{0_{j,i}} + r_{1_{j,i}}x_{m+1} + \cdots + r_{n_{j,i}}x_{m+n})
 \end{equation}
is as in  (\ref{sigmadefinicion2}) of Definition \ref{def.skewpbwextensions}. As $A$ is graded skew PBW extension then it is homogeneous quadratic, since $r_1,\dots, r_s$, $f_{hk}$, $g_{ji}$ are homogeneous polynomials of degree two in $\mathbb{K}\langle t_1, \dots, t_m, x_1, \dots, x_n\rangle$. Now, let $S^{(1)}_{tx}:=\{1,\dots, m, m+1, \dots, m+n\}$. From the relations  (\ref{eq.rel xt}) we obtain the set $S_{tx}:= \{(k,l)\mid 1\leq k\leq m, m+1\leq l\leq m+n\}$. From the relations (\ref{eq.rel xx})  we obtain the set $S_{x}:= \{(m+i,m+j)\mid 1 \leq i\leq j\leq n)\}$. From Definition \ref{def.skewpbwextensions}, we have that $R\subseteq A$ and  $A$ is a left free $R$-module. Then, for the $\mathbb{K}$-algebra $A$, we have that

 \begin{equation*}
  S^{(p)}= \{(i_1,\dots, i_k, i_{k+1},\dots, i_p)\mid (i_1,\dots, i_k)\in S_t^{(k)} \text{ and } i_{k+1}\leq \cdots\leq i_p\}\\.
 \end{equation*}
 So,
 \begin{equation}
 (A,S):= \{t_{i_1}\cdots t_{i_k}x_{i_{k+1}}\cdots x_{i_p} \mid (i_1,\dots, i_k, i_{k+1},\dots, i_p)\in \cup_{p> 0}S^{(p)}\}
\end{equation}
span $A$  as a vector space.
As $(R,S_t):= \{t_{i_1}\cdots t_{i_v}\mid   (i_1,\dots, i_{\nu}) \in \cup_{p> 0}S_t^{(p)} \}$ is a $\mathbb{K}$-basis for $R$ and $A$ is a left free $R$-module, with basis the basic elements
\begin{multline*}
\{x^{\alpha}=x_{m+1}^{\alpha_{m+1}}\cdots
x_{m+n}^{\alpha_{m+n}}\mid \alpha=(\alpha_{m+1},\dots ,\alpha_{m+n})\in
\mathbb{N}^n\}\\= \{x_{i_{k+1}}\cdots x_{i_p}\mid  m+1\leq i_{k+1}\leq \cdots\leq i_p\leq m+n\}\cup \{1\},
\end{multline*}
then $(A,S)$   is a PBW basis of $A$.  Therefore $A$  is a PBW algebra.\\

\end{proof}

\begin{remark}\label{rem.orden in PBW}
If in the free algebra $\mathbb{K}\langle x_1, \dots, x_n\rangle$ we fix the set $\{1,2,\dots,
n\}$, we implicitly understand that $x_1<x_2<\cdots<x_n$. For example, for $A = \mathbb{K}\langle
x, y, z\rangle/\langle z^2-xy-yx, zx-xz, zy-yz\rangle$   with $x<y<z$, i.e., $x=x_1$, $y=x_2$,
$z=x_3$, we have that $S^{(1)}=\{1,2,3\}$, $S=\{(1,1), (1,2), (1,3), (2,1), (2,2), (2,3)\}=
S^{(2)}$. Note that $(A,S)$ is not a $\mathbb{K}$-basis for $A$. Indeed: $(2,1,1)$, $(1,1,2) \in
S^{(3)}$ and therefore the classes (nonzero)  of $yx^2$, $x^2y\in (A,S)$, but $yx^2-x^2y= yx^2 +
xyx-x^2y-xyx=(xy+yx)x-x(xy+yx) = z^2x-xz^2=0$, since $xz=zx$ in $A$. Because of $A =
\mathbb{K}\langle x, y, z\rangle/\langle z^2-xy-yx, zx-xz, zy-yz\rangle\cong
\sigma(\mathbb{K}[z]\langle x,y\rangle$ is a graded skew PBW extension of the PBW algebra
$\mathbb{K}[z]$, in this case the Proposition \ref{prop.R PBW impl A PBW} fails. So it is important
the order of the generators of the free algebra $L$ as in the proof of the Proposition \ref{prop.R
PBW impl A PBW}; for the  graded skew PBW extension  $A = \sigma(\mathbb{K}[z]\langle x,y\rangle$
we have that $A=\mathbb{K}\langle z, x, y\rangle/\langle z^2-xy-yx, zx-xz, zy-yz\rangle$, i.e.,
$z=x_1<x=x_2<y=x_3$. In this case we write the relations as $yx= -xy + z^2; xz=zx; yz=zy$, whereby
$(3,2)$, $(2,1)$, $(3,1)\notin S$. So, $S=\{(1,1),(1,2),(1,3),(2,2),(2,3),(3,3)\}$,
$S^{(p)}=\{(i_1,i_2, \dots, i_p)\mid i_1\leq i_2\leq\cdots\leq i_p\}$ and
$(A,S)=\{z^{\alpha_1}x^{\alpha_2}y^{\alpha_3}\mid \alpha_1, \alpha_2, \alpha_3\geq 0\}$ is a PBW
base for $A$.
\end{remark}

\begin{theorem}[\cite{Priddy1970}, Theorem 5.3 ]\label{teo.PBW alg is Koszul} 
 If $B$ is a PBW algebra then $B$ is a Koszul algebra.
  \end{theorem}
The proof of the previous theorem can be also found in \cite{Polishchuk}, Theorem 3.1, page 84;
they also exhibit an example of a Koszul algebra which is not a PBW algebra.

\begin{corollary}\label{cor. R PBW impl A homo}

Let $A$ be a graded skew PBW extension of  a finitely presented algebra  $R$. If  $R$ is a PBW
algebra then $A$ is Koszul algebra.
\end{corollary}
\begin{proof}
From Proposition \ref{prop.R PBW impl A PBW}  and Theorem \ref{teo.PBW alg is Koszul}.
\end{proof}

\begin{example} Let $R=\mathbb{K}[ t_1, \dots, t_m]$ be the classical polynomial ring. Then from Corollary \ref{cor. R PBW impl A homo}
every graded skew PBW extension of $R$ is  Koszul. Therefore, Examples \ref{ex.grad skew pbw ext}
are Koszul algebras. Also, by Remark  \ref{rem.orden in PBW} and Corollary \ref{cor. R PBW impl A
homo}, we have that   $A=\mathbb{K}\langle z, x, y\rangle/\langle z^2-xy-yx, zx-xz, zy-yz\rangle$
is a Koszul algebra. Note that $A=\mathbb{K}\langle z, x, y\rangle/\langle z^2-xy-yx, zx-xz,
zy-yz\rangle=\sigma(\mathbb{K}[z]\langle x, y\rangle= \mathbb{K}[z][x; \sigma_1,\delta_1][y;
\sigma_2,\delta_2]$ is a graded iterated  Ore extension, where $\sigma_1(z)=z$, $\sigma_2(x)=-x$,
$\delta_1(z)=0$ and $\delta_2(x)=z^2$. So, we also can be use the Proposition \ref{iterated Ore} to
guarantee that $A$ is Koszul.
\end{example}

\begin{remark}

(i) Some of the algebras in Example \ref{ex.grad skew pbw ext} had already been presented by other
authors as Koszul algebras using other characterizations. For example, Smith in \cite{Smith2},
Proposition 12.1, showed that the homogenized enveloping algebra $\mathcal{A}(\mathcal{G})$ is
Koszul.

(ii) The converse of Corollary \ref{cor. R PBW impl A homo} is false.
Indeed, the $\mathbb{K}$-algebra $R$ minimal in the numbers of generators and relations for
algebraically closed field $\mathbb{K}$ with relations $x^2+yz=0$ and $x^2+azy=0$, $a\neq 0,1$, is
Koszul but $R$ is not a PBW algebra (see \cite{Polishchuk}, Example of page 84). The associated
graded Ore extension $A:=R[u]$ is Koszul  algebra (\cite{Phan}, Corollary 1.3) and graded skew PBW
extension.

(iii) Let $R$ as in the part (ii) above. Note that $A=R[u]\cong \mathbb{K}\langle x,y,z,u\rangle/\langle x^2+yz, x^2+azy, ux-xu,uy-yu,uz-zu\rangle$,
with $a\neq 0,1$. So, $x<y<z<u$ and \linebreak $S=\{(1,1),(1,2), (1,3),(1,4), (2,1),(2,2), (2,4)\}$. Therefore
$(1,1,2)$, $(2,1,1)\in S^{(3)}$ and $x^2y$, $yx^2$ are nonzero monomials in $A$, but
$a^{-1}yx^2+x^2y= yzy- yzy=0$. Then $(A,S)$ is not a PBW basis, i.e., $A$ is not a PBW algebra. So, if  $A$ is a  graded skew PBW
extension of the Koszul algebra $R$  does not imply that $A$ is PBW algebra.

(iv) With the above reasoning we have that not any graded skew PBW extension is a $PBW$ algebra.

(v) We have also that not all graded skew PBW extension are Koszul. Indeed, let $R=\mathbb{K}\langle
x,y\rangle/\langle y^2-xy, y^2\rangle$ be a homogeneous quadratic non-Koszul algebra
(\cite{Cassidy2008}, page 10), then $R[u]$ is a non-Koszul  associate graded Ore extension of $R$,
which is also a graded skew PBW extension.

\end{remark}

\section{Lattices}

A \emph{lattice} is a discrete set $\Omega$ endowed with two idempotent (i.e., $a \cdot a = a$)
commutative, and associative binary operations $\wedge, \vee: \Omega \times \Omega\to \Omega$
satisfying the following \emph{absorption identities}: $a \wedge (a \vee b)= a$, $(a\wedge b)\vee
b=b$. A lattice is called \emph{distributive} if it satisfies the following distributivity
identity: $a\wedge (b \vee c) = (a\wedge b) \vee (a\wedge c)$. Let $W$ be a vector space. The set
$\Omega_W$ of all its linear subspaces is a lattice with respect to the operations of sum and
intersection. Given $X_1, \dots, X_z$ subspaces of a vector space $W$, we may consider the
\emph{sublattice} of subspaces of $W$ \emph{generated} by $X_1,\dots, X_z$ by the operations of
intersection and summation.  We will say that a \emph{collection of subspaces} $X_1,\dots, X_z
\subseteq W$ \emph{is distributive} if it generates a distributive lattice of subspaces of $W$.

\begin{proposition}[\cite{Polishchuk}, Proposition 1-7.1]\label{prop.distr base} Let $W$ be a vector space and $X_1, \dots, X_z\subseteq W$ be a collection
of its subspaces. Then the following conditions are equivalent:
\begin{enumerate}
\item[\rm(i)] the collection $X_1,\dots, X_z$ is distributive;
\item[\rm(ii)] there exists a direct sum decomposition $W=\bigoplus_{j\in J}W_j$ of the vector space $W$ such that each of the subspaces $X_i$ is the sum of a set of subspaces $W_j$.
\item[\rm(iii)] there exists a basis $\mathcal{B}=\{w_i\ \mid\ i\in I \}$ of the vector space $W$ such that each of
the subspaces $X_i$ is the linear span of a set of vectors $w_i$.
\item [\rm(iv)] there exists a basis $\mathcal{B}$ of the vector space $W$ such that  $\mathcal{B}\cap X_i$ is a basis  of
the subspace $X_i$ for each $1\leq i\leq z$ (\cite{Backelin1}, Lemma 1.2).
\end{enumerate}
\end{proposition}

Let $A=\mathbb{K}\langle x_1,\dots, x_n\rangle/I$, where $I$ is a two-sided ideal generated by
homogeneous elements and let $A_+ =\bigoplus_{p>0}A_p$. The \emph{lattice associated} to $A$,
$\Omega(A)$ is the  lattice generated by $\{A_+^{\lambda}I^{\mu}A_+^{\nu}\ \mid \
\lambda,\mu,\nu\geq 0\}\subseteq \{\text{ Subspaces of } \mathbb{K}\langle x_1,\dots,
x_n\rangle\}$, where $I^0= \mathbb{K}\langle x_1,\dots, x_n\rangle$, $I^1=I$; $I^2=\{\sum xy\ \mid
\ x,y\in I\}$, $\cdots$. The lattice generated by $\{A_+^{\lambda}I^{\mu}A_+^{\nu}\ \mid \
\lambda,\mu,\nu\geq 0,  \lambda +\mu +\nu=j\}$ is denoted by $\Omega_j(A)$. Backelin in
\cite{Backelin1} shows that  $A$ is Koszul if and only if $A$ is quadratic and $\Omega(A)$ is
distributive and that $\Omega(A)$ is distributive if and only if for all $j\geq 2$, $\Omega_j(A)$
is distributive. So, $A$ is Koszul if and only if $A$ is quadratic and $\Omega_j(A)$ is
distributive,  for all $j$. As a consequence of this, Polishchuk and Positselski in
\cite{Polishchuk} show the following criteria for Koszulness.

\begin{lemma}[\cite{Polishchuk}, Theorem 2-4.1]\label{lemma.Kosz iff distr} A homogeneous quadratic algebra
$A=L/\langle P\rangle$ $(L=\mathbb{K}\langle x_1,\dots, x_n\rangle)$ is Koszul if and only if for
all $k\geq 0$, the collection of subspaces
\begin{equation}\label{eq.Xi}
X_i:=  L_{i-1}PL_{k-i-1}\subset L_k,\ i=1,\dots, k-1
\end{equation}
is distributive.
\end{lemma}

Let $R=\mathbb{K}\langle t_1, \dots, t_n\rangle/I$ be a homogeneous quadratic algebra. Note that
$\Omega(R)$  only depend on the subalgebra of $R$ generated by those of the generators of $R$ which
\textquotedblleft appear\textquotedblright\ in a set of minimal relations for $R$.

\begin{lemma}[\cite{Backelin1}, Lemma 2.3]\label{lema. R distr sii A dist}
Let $R=\mathbb{K}\langle t_1, \dots, t_n\rangle/I$ be a quadratic algebra and let
$$A=\mathbb{K}\langle t_1,\dots, t_m, x_1,\dots, x_n\rangle/\langle I\rangle,$$ where $\langle
I\rangle$ is the two-sided ideal of $\mathbb{K}\langle t_1,\dots, t_m, x_1,\dots, x_n\rangle$
generated by $I$. Then  $\Omega(R)$ is distributive if and only if $\Omega(A)$ is distributive.
\end{lemma}

\begin{lemma}\label{lemma.R Kosz iff exten Koszul} A quadratic algebra $R=\mathbb{K}\langle t_1, \dots, t_m\rangle/I$ is Koszul if and only if $$A=\mathbb{K}\langle t_1,\dots, t_m, x_1,\dots, x_n\rangle/\langle I\rangle$$ is Koszul, where $\langle I\rangle$ is the two-sided ideal of $\mathbb{K}\langle t_1,\dots, t_m, x_1,\dots, x_n\rangle$ generated by $I$.
\end{lemma}

\begin{proof} Note that $R$ is quadratic if and only if $A$ is quadratic.  Also, by Lemma \ref{lema. R distr sii A dist},  $\Omega(R)$ is distributive if and only if  $\Omega(A)$ is distributive. Therefore, by Lemma \ref{lemma.Kosz iff distr}, $R$ is Koszul if and only if $A$ is Koszul.
\end{proof}

Related to Proposition \ref{iterated Ore} we have the following theorem.

\begin{theorem}
If $A=\sigma(R)\langle x_1, \dots, x_n\rangle$ is a graded skew PBW extension of a finitely
presented Koszul algebra $R$, then  $A$ is Koszul.
\end{theorem}

\begin{proof}
 Let $R$ be a  finitely presented  algebra; by Proposition \ref{prop.Lez fga},
\begin{equation}\label{eq.repr R}
R= \mathbb{K}\langle t_1,\dots, t_m\rangle/\langle P\rangle
\end{equation}
where $P$ is the $\mathbb{K}$-space generated by homogeneous polynomials
\begin{equation}\label{eq.relat R}
r_1, \dots, r_s \in L^R:=\mathbb{K}\langle t_1,\dots, t_m\rangle.
\end{equation}
Let $A=\sigma(R)\langle x_1, \dots, x_n\rangle$ be a graded skew PBW extension. Then by Remark
\ref{rem.prop of graded skew}, $A$ is a finitely presented algebra. So, by Proposition
\ref{prop.Lez fga},
\begin{equation}\label{eq.repr A}
A=\mathbb{K}\langle t_1,\dots, t_m, x_1, \dots, x_n \rangle/\langle W \rangle,
\end{equation}
where $W$ is the $\mathbb{K}$-space generated by the polynomials
\begin{multline}\label{eq.relat A}
r_1, \dots, r_s, \ x_jt_k-\sigma_j(t_k)x_i-\delta_j(t_k), \ x_jx_i-c_{i,j}x_{i}x_{j} - (r_{0_{j,i}} + r_{1_{j,i}}x_{1} + \cdots + r_{n_{j,i}}x_{n})\\
 \in L:=\mathbb{K}\langle t_1,\dots, t_m, x_1, \dots, x_n \rangle,
\end{multline}
with $1\leq i,j\leq n$, $1\leq k\leq m$.\\

 Since $R$ is a Koszul algebra then:
 \begin{enumerate}
 \item[\rm(i)] By Lemma \ref{lemma.Kosz iff distr} we have that $R$ is  homogeneous quadratic algebra, and by Remark \ref{rem.prop of graded skew}, $A$ is homogeneous quadratic algebra.
 \item[\rm(ii)] By Lemma \ref{lemma.R Kosz iff exten Koszul}, we have that $A_P:= \mathbb{K}\langle t_1, \dots, t_m, x_1, \dots, x_n\rangle/\langle P\rangle_X$ is Koszul, where $\langle P\rangle_X$ is the two-sided ideal of $\mathbb{K}\langle t_1, \dots, t_m, x_1, \dots, x_n\rangle$ generated by the polynomials as in \ref{eq.relat R}. So, by Lemma  \ref{lemma.Kosz iff distr}, we have that for all $k\geq 0$, the collection of subspaces
     \begin{equation}\label{eq.X_i^P}
     X_i^P:= L_{i-1}PL_{k-i-1}\subseteq L_k, \ i=1,\dots, k-1
     \end{equation}
     is distributive. Therefore, by Proposition \ref{prop.distr base}, there exist a basis $\mathcal{B}_k$ of the  space  $L_k$ such that $\mathcal{B}_k\cap X_i^P$ is a basis of $X_i^P$ for each $1 \leq i \leq k-1$. Let $X_i:=  L_{i-1}WL_{k-i-1}\subseteq L_k$, $i=1,\dots, k-1$, where $W$ is the space generated by the polynomials as in (\ref{eq.relat A}).
 \end{enumerate}
 Let $Y:= (X_i\ \backslash \ X_i^P)\cap \mathcal{B}_k$. Since $X_i^P$ is a subspace of $X_i$ we
 claim that $\bar{Y}:=\{y + X_i^P\ \mid \ y\in Y\}=\{\bar{y}\in X_i/X_i^P \ \mid \ y\in Y \}$ is a
 basis of $X_i/X_i^P$. Indeed: if $0\neq\bar{x}\in X_i/X_i^P$, then $\bar{x}= x+X_i^P$, with $x\in X_i \ \backslash \ X_i^P$. Note
 that $x=k_1b_1+\cdots + k_{\rho}b_{\rho}$, where  $b_1,\dots, b_{\rho}$ are different nonzero elements in $\mathcal{B}_k$
 and $k_1,\dots, k_{\rho}\in \mathbb{K}$. Then $\bar{x}=\overline{k_1b_1+\cdots + k_{\rho}b_{\rho}}=k_1\bar{b_1}+\cdots + k_{\rho}\bar{b_{\rho}}$. If $b_\nu\in X_i^P$
 for some $1\leq \nu\leq \rho$ then $\bar{b_\nu}=0$. So  $\bar{x} = s_1\bar{v_1}+\cdots + s_{\mu}\bar{v_{\mu}}$ with  $s_1, \dots, s_{\mu}\in \mathbb{K}$
 and  $v_1,\dots, v_{\mu}\in Y$. Now suppose that $k_1\bar{y_1}+\cdots + k_v\bar{y_v}=\overline{0}$  with $k_1, \dots, k_v\in \mathbb{K}$
 and $0\neq\bar{y_1}$, $\dots$, $0\neq\bar{y_v}\in \bar{Y}$. Then ${y_1},\dots, {y_v}\notin X_i^P$,  $\overline{k_1y_1+\cdots + k_vy_v}=\overline{0}$
 and so $k_1y_1+\cdots + k_vy_v\in X_i^P$. As  $X_i^P\cap \mathcal{B}_k$ is a basis of $X_i^P$ then there are different nonzero
 elements $w_{v+1},\dots, w_{v+\mu} \in X_i^P\cap \mathcal{B}_k$ such that $k_1y_1+\cdots + k_vy_v=k_{v+1}w_{v+1} +\dots + k_{v+\mu}w_{v+\mu}$,
 with $k_{v+1},\dots, k_{v+\mu}\in \mathbb{K}$. As ${y_1},\dots, {y_v}\notin X_i^P$ then ${y_1}, \dots, {y_v}, w_{v+1}, \dots , w_{v+\mu}$
 are nonzero different elements in $\mathcal{B}_k$ such that $k_1y_1+\cdots + k_vy_v+(-k_{v+1})w_{v+1}+ \cdots + (-k_{v+\mu})w_{v+\mu}=0$.
 Then $k_1=\cdots = k_v= k_{v+1}=\cdots = k_{v+\mu}=0$.

Therefore, by Theorem 3.33 in \cite{Spindler}, we have that $(\mathcal{B}_k\cap X_i^P)\cup Y=
\mathcal{B}_k\cap X_i$ is a basis of $X_i$. So, by Proposition \ref{prop.distr base} the collection
of subspaces $X_1, \dots, X_{k-1}$ is distributive for each $k\geq 0$. Whence, by Lemma
\ref{lemma.Kosz iff distr} we have that $A$ is Koszul.

\end{proof}

\end{document}